\documentclass[11pt]{amsart}
\usepackage{amsmath}
\usepackage{amssymb}
\usepackage{amscd}
\usepackage{comment}
\usepackage{enumerate}

\def\NZQ{\mathbb}               
\def\NN{{\NZQ N}}

\def\ZZ{{\NZQ Z}}

\newcommand{\choos}[2]{\begin{pmatrix} 
#1 \\
#2 
\end{pmatrix}}

\newtheorem{Theorem}{Theorem}[section]
\newtheorem{Lemma}[Theorem]{Lemma}
\newtheorem{Corollary}[Theorem]{Corollary}
\newtheorem{Proposition}[Theorem]{Proposition}
\newtheorem{Remark}[Theorem]{Remark}

\newtheorem{Definition}[Theorem]{Definition}

\newtheorem{Question}[Theorem]{Question}

\newtheorem*{theorem*}{Theorem}
\newtheorem*{lemma*}{Lemma}
\newtheorem*{corollary*}{Corollary}

\let\epsilon\varepsilon
\let\phi=\varphi
\let\kappa=\varkappa

\begin{document}
\author{Soumya Deepta Sanyal}
\title{Irrational behavior of algebraic discrete valuations}

\date{}
\address{28 Mathematical Sciences Building, Department of Mathematics,
University of Missouri, Columbia, MO 65211, USA}
\email{sds2p8@mail.missouri.edu}


\maketitle

\section{Introduction}

In this paper, we discuss the behavior of algebraic discrete valuations dominating a two dimensional normal (Noetherian) local domain $(R,\mathfrak{m}_R)$. \\

To a valuation $\nu$ of the quotient field $K=QF(R)$ of $R$, we associate a valuation ring $V_{\nu}=\{f \in K \mid \nu(f) \geq 0 \}$. This ring is quasi-local, with maximal ideal $\mathfrak{m}_{\nu} = \{f \in K \mid \nu(f)>0\}$. We say that the valuation $\nu$ dominates $(R,\mathfrak{m}_R)$ if $V_{\nu} \supset R$ and $R \cap \mathfrak{m}_{\nu} = \mathfrak{m}_R$. We say that the valuation $\nu$ is discrete if its value group is order isomorphic to $\mathbb{Z}$. In this case $V_{\nu}$ is Noetherian, and thus is local.\\

For $n \in \mathbb{Z}_{\geq 0}$, let $I_n=\{f \in R \mid \nu(f) \geq n\}$. The questions we address in this paper are: what is $\ell_R(R/I_n)$ for $n \gg 0$? How close is it to being a polynomial? What is $\lim_{n \to \infty} \frac{\ell_R(R/I_n)}{n^2}$? Is this limit well behaved?\\   

There are two types of discrete valuations $\nu$ dominating $R$. The first case is that $\nu$ is divisorial, i.e., the residue field extension $V_{\nu}/\mathfrak{m}_{\nu} : R/\mathfrak{m}_R$ is transcendental. By Abhyankar's inequality (Theorem 1 of \cite{Abhyankar}) it follows that the residue field extension given by a divisorial valuation $\nu$ dominating a normal local domain of dimension two is of transcendence degree exactly one. The second case is that $\nu$ is algebraic, i.e., the residue field extension $V_{\nu}/\mathfrak{m}_{\nu} : R/\mathfrak{m}_R$ is algebraic.\\

In the divisorial case, it is known by work of Cutkosky and Srinivas in \cite{DaleSrinivas} that when $R$ is excellent and equicharacteristic of dimension 2, $\ell_R(R/I_n)$ can be written as a quadratic polynomial $Q(n)$ plus a bounded function $\sigma(n)$ for $n\gg 0$. They further show that if $R/\mathfrak{m}_{R}$ has characteristic zero, or is finite, then $\sigma(n)$ is periodic for $n\gg 0$ and they give an example to show that there exist $R$, $\nu$ with characteristic $R/\mathfrak{m}_{R}>0$ such that $\sigma(n)$ is \textbf{not} eventually periodic. Finally, in their Example 6, they give an example of a divisorial valuation $\nu$ dominating a 3 dimensional ring such that $\lim_{n \rightarrow \infty}{\frac{\ell_R(R/I_n)}{n^3}}$ is an irrational number.\\

The analysis in dimension two of  \cite{DaleSrinivas} left open the remaining case that $\nu$ is an algebraic discrete valuation. In the case that $V_{\nu}/\mathfrak{m}_{\nu}$ is finite over $R/\mathfrak{m}_R$, everything is well behaved. In section \ref{finitesection} of this paper, we prove the following Theorem \ref{linear} and Corollary \ref{zerolimit}:\\

 \begin{Theorem}\label{linear}
Suppose that $\nu$ is an algebraic discrete valuation dominating a local domain $(R,\mathfrak{m}_R)$, and that $V_{\nu}/\mathfrak{m}_{\nu}$ is finite over $R/\mathfrak{m}_R$. Then there exist $c,b \in \mathbb{Z}$ such that $$\ell_R(R/I_n)=cn+b$$ for $n\gg 0$, where $c=|V_{\nu}/\mathfrak{m}_{\nu}:R/\mathfrak{m}_R|$.
\end{Theorem}


 \begin{Corollary}\label{zerolimit}
Suppose that $\nu$ is an algebraic discrete valuation dominating a local domain $(R,\mathfrak{m}_R)$, and that $V_{\nu}/\mathfrak{m}_{\nu}$ is finite over $R/\mathfrak{m}_R$. Then $$\lim_{n \to \infty} \frac{\ell_R(R/I_n)}{n^2}=0.$$\\
\end{Corollary}

In section \ref{infinitesection} of this paper, we show in Theorems \ref{example} and \ref{secondset_of_limits}, and Corollary \ref{secondirrational_limits}, that when we consider infinite algebraic residue field extensions, we have a more interesting result:\\

\begin{Theorem} \label{firstnon-poly}
There exists a regular local ring $(R,\mathfrak{m}_R)$ of dimension 2 and a discrete, rank 1 valuation $\nu$ of the quotient field of $R$ dominating $R$, such that the function $\ell_R(R/I_n)$ cannot be written as a quasi-polynomial plus a bounded function for large integers $n$. \\ 
\end{Theorem}

We remind the reader that a quasi-polynomial in the indeterminate $n$ is an expression of the form $a_d(n)n^d + a_{d-1}(n)n^{d-1} + \dots + a_1(n)n + a_0(n)$, where $d$ is the degree of the quasi-polynomial, and the coefficients $a_d(n), a_{d-1}(n), \dots, a_1(n), a_0(n)$ are periodic functions of $n$.\\

\begin{Theorem} \label{firstset_of_limits}
Let $C \in \mathbb{R} \cap [0,\frac{1}{2}]$ be given. There exists a regular local ring $(R,\mathfrak{m}_R)$ of dimension 2 and a discrete, rank 1 valuation $\nu$ of the quotient field of $R$ dominating $R$, such that each of the functions $\ell_R(R/I_n)$ and $\ell_R(I_n/I_{n+1})$ is not a quasi-polynomial plus a bounded function for large integers $n$ and such that $\lim_{n \rightarrow \infty} {\frac{\ell_R(I_n/I_{n+1})}{n}} = C$ and $\lim_{n \rightarrow \infty} {\frac{\ell_R(R/I_{n})}{n^2}} = \frac{C}{2}$.\\ 
\end{Theorem}

\begin{Corollary} \label{firstirrational_limits}
There exists a regular local ring $(R,\mathfrak{m}_R)$ of dimension 2 and a discrete, rank 1 valuation $\nu$ of the quotient field of $R$ dominating $R$, such that each of the functions $\ell_R(R/I_n)$ and $\ell_R(I_n/I_{n+1})$ is not a quasi-polynomial plus a bounded function for large integers $n$ and such that each of $\lim_{n \rightarrow \infty} {\frac{\ell_R(I_n/I_{n+1})}{n}}$ and $\lim_{n \rightarrow \infty} {\frac{\ell_R(R/I_{n})}{n^2}}$ are irrational (even transcendental) positive numbers.\\  
\end{Corollary}


%

%

%
  



Next, we discuss some questions that arise from the methods used in this paper. Since $\{I_n \}$ is a graded family of $\mathfrak{m}_R$-primary ideals in the two dimensional local ring $R$, the limit $$\lim_{n \to \infty} \frac{\ell_R(R/I_n)}{n^2} $$ is known to exist in equicharacteristic regular local rings by Mustata \cite{Mustata}, and they exist in arbitrary regular local rings (even analytically irreducible local rings) by Cutkosky \cite{Dalemult} and \cite{Dale_final_mult}. In light of these results, we may ask the following question:\\

 \begin{Question}
  Which numbers $C$ are realizable as limits $$\lim_{n \to \infty}{\frac{\ell_R(R/I_n)}{n^2}} $$ from an algebraic discrete valuation on a regular local ring $R$ of dimension two?\\
 \end{Question}

In this paper, we show that all the numbers in the real interval $[0,\frac{1}{4}]$ are realizable as limits. Since $I_1 = \mathfrak{m}_R$, and $I_1^n \subset I_n$ for all $n$, by comparison with the Hilbert-Samuel multiplicity we have the upper bound $$\lim_{n \to \infty}{\frac{\ell_R(R/I_n)}{n^2}} \leq \frac{e(R)}{2} = \frac{1}{2}.$$

  In all of our examples, we have the stronger result that 
  
  \begin{equation} \label{alphalimit}
  \lim_{n \to \infty}{\frac{\ell_R(I_n/I_{n+1})}{n}}  
  \end{equation}
exists, and every number in the real interval $[0,\frac{1}{2}]$ can be obtained as a limit.\\
  
Thus we may ask the following question:\\
  
  \begin{Question} \label{question_limit}
   Does $\lim_{n \to \infty}{\frac{\ell_R(I_n/I_{n+1})}{n}}$ always exist for an algebraic discrete valuation dominating a regular local ring of dimension two?\\
  \end{Question}

It is known that 
this limit does not generally exist if $\{I_n\}$ is a filtration of $\mathfrak{m}_R$-primary ideals (Theorem 4.6 of \cite{Dale_asymptotic_mult}), so a positive answer to Question \ref{question_limit} would use special properties of the valuation ideals $I_n$. \\
 
Now we briefly discuss the methods used in this paper. In constructing our examples, we make use of an algorithm of \cite{Dale} that generalizes an algorithm of \cite{Spivakovsky} for constructing generating sequences of valuations. The algorithm in \cite{Dale} is valid in arbitrary two dimensional regular local rings.  This technique of generating sequences is also used in \cite{GHK} and \cite{GK} to find stable toric forms of extensions of associated graded rings along a valuation in finite defectless extensions of algebraic function fields of dimension two. \\

We take as our ground field $k=L\left( \left\{ \sigma_i \right\} \right)$, where $L$ is an arbitrary field and $\left\{\sigma_i\right\}_{i \in \mathbb{Z}_+}$ is a set of algebraically independent elements over $L$.  We take for our ring $R:=k[[x,y]]$  and inductively define a generating sequence of our valuation, $\{P_i\}_{i \geq 0} \subset R$ by $P_0=x, P_1=y, \text{ and } P_{i+1}=P_i^2-\sigma_ix^{2r_i} \text{ for } i \geq 1.$  Provided that the $r_i \in \mathbb{N}$ are such that $r_0=1 \text{ and } r_{i+1} > 2r_i \text{ for } i \geq 1,$ the algorithm ensures that there exists a unique discrete valuation $\nu$ of $QF(R)$ dominating $R$, such that $\nu(P_i)=r_i \text{ for } i \in \mathbb{N}.$ The residue field of $\nu$ is naturally isomorphic to $k(\{\sqrt{\sigma_i}\})$.\\

In the associated graded algebra $\bigoplus_{s \geq 0} I_s/I_{s+1} ,$ the lengths of the graded components $I_s/I_{s+1}$ are given by the number of distinct monomials in the generators $\{P_i\}$ that have value equal to $s$.\\

 That is, a $k$-basis of $I_s/I_{s+1}$ is:\\ 

\begin{multline} $$\mathcal{B}_s=\{[P_0^{n_0}P_1^{n_1}\cdots P_i^{n_i}] \mid i \in \mathbb{N}, n_0,\ldots, n_i \in \mathbb{N}, \\ n_j \in \{0,1\} \text{ for } 1 \leq j \leq i, \text{ and } n_0+n_1r_1+\cdots + n_ir_i = s\}.$$ \end{multline}  \\

 By multiplying basis elements through by an element of suitable $\nu$-value in $R$ and checking linear independence conditions, we show that $\ell_R(I_s/I_{s+1})$ is a non-decreasing function of $s$. By choosing the $\nu$-values $r_i$ of the generators, we can control the growth of $\ell_R(I_s/I_{s+1})$. By imposing a combinatorial condition on the sequence $\{r_i\}$, that $r_{i+1} > r_0 + r_1 + \cdots + r_i$ for all $i \geq 0$, we can ensure that there is a partition of $\mathbb{R}_{\geq 0}$ of the form $\{[a_k,b_k)\}_{k \in \mathbb{Z}_{>0}}$, such that on every interval $[a_k,b_k)$, the function $\ell_R(I_s/I_{s+1})$ is constant, and takes on distinct values for every $k$. This part of the construction already shows that the Hilbert function is highly non-polynomial.  \\

%


%


Next, we show that for these examples, the limit $\frac{\ell_R(I_n/I_{n+1})}{n}$ exists, and deduce the existence of $\frac{\ell_R(R/I_{n})}{n^2}$ as a consequence. \\

 The construction places the set of multiplicities $(0,\frac{1}{4}]$ in bijective correspondence with $[0,\infty)$. This is done by specifying the values of $r_{i+1} - (2r_i + 1)$.  More precisely, to construct the valuation to have multiplicity $B \in (0,\frac{1}{4}]$, we set the difference $r_{i+1} - (2r_i + 1)$ to be equal to the coefficient of $\frac{1}{2^{i-1}}$ in the $2$-adic expansion of $\frac{2}{B} - 2$.  For the remaining case, taking $r_{i+1} - (2r_i + 1)=2^{i-1}$ for $i \geq 0$ gives an example where the multiplicity is $0$. \\

\section{Notation}

Let $\NN$ denote the set $\{0,1,2,\dots \}$ and $\ZZ_+$ denote the set $\{1,2,3,\dots \}$. Suppose that $\nu$ is a discrete valuation dominating a local ring $(R,\mathfrak{m}_R)$. Let $I_n = \{f \in R \mid \nu(f) \geq n \}$. If $f \in I_n$, denote by $[f]$ its image in $I_n/I_{n+1}$. 
If $r,s \in \NN$, $[x] \in I_r/I_{r+1}$ and $[y] \in I_s/I_{s+1}$, denote by $[x]\cdot[y]$ their product $(x+I_{r+1})(y+I_{s+1}) \in I_{r+s}/I_{r+s+1}$  in the associated graded ring $\oplus_{n \in \NN}{I_n/I_{n+1}}$. In particular if $c \in R/\mathfrak{m}_R$, we will write $c\cdot[x] \in I_r/I_{r+1}$. Also, if $c \in R/\mathfrak{m}_R$ and $[d] \in V_{\nu}/\mathfrak{m}_{\nu}$, then $c\cdot[d] = (c+\mathfrak{m}_R)(d+\mathfrak{m}_{\nu}) \in V_{\nu}/\mathfrak{m}_{\nu}$. If $R/\mathfrak{m}_R \cong k \subset R$, then $\ell_R(I_s/I_{s+1})=dim_{R/\mathfrak{m}_R}(I_s/I_{s+1})$. In this case (which applies to this paper), we will use the two notations interchangeably. Finally, we let $\Gamma_{\nu}=\{\nu(f) \mid f\ \in QF(R) \setminus \{0\} \}$ and  $S^R(\nu) = \{\nu(f) \mid f \in R\ \setminus \{0\} \} \subset \Gamma_{\nu} $ denote the value group of $\nu$ and the value semigroup of $\nu$ respectively. \\ 

\section{Results for finite residue field extensions} \label{finitesection}

In this section, we prove Theorem \ref{linear} and Corollary \ref{zerolimit} from the introduction.\\ 

\begin{Lemma} \label{persist}
Suppose that $\nu$ is an algebraic discrete valuation dominating a local domain $(R,\mathfrak{m}_R)$. Let $n \in \mathbb{N}$. Suppose that $dim_{R/\mathfrak{m}_R}(I_n/I_{n+1})=r$. Then for all $k \in S^R(\nu)$, we have that $dim_{R/\mathfrak{m}_R}(I_{n+k}/I_{n+k+1}) \geq r$.
\end{Lemma}

\begin{proof}
Suppose that $e_1, \dots, e_r$ are elements of $I_n \setminus I_{n+1}$ such that $\left[e_1\right], \dots, \left[e_r\right]$ form a basis for $I_n/I_{n+1}$ over $R/\mathfrak{m}_R$. Let $k \in S^R(\nu)$ be given. Then there exists $g \in R$ such that $\nu(g)=k$. Hence $ge_1, \dots, ge_r$  are elements of $I_{n+k} \setminus I_{n+k+1}$ with nonzero residues $\left[ge_1 \right], \dots, \left[ ge_r \right]$ in $I_{n+k}/I_{n+k+1}$. \\

Suppose that $\sum_{i=1}^r{c_i\left[ge_i\right]}=\left[0 \right]$ in $I_{n+k}/I_{n+k+1}$, where $c_i \in R/\mathfrak{m}_R$ for $1 \leq i \leq r$. Then $\nu(\sum_{i=1}^r{c_ige_i})>n+k$. Thus $\nu(\sum_{i=1}^r{c_ie_i})=\nu(\sum_{i=1}^r{c_ige_i})-\nu(g)>n+k-k=n$, whence $\sum_{i=1}^r{c_i\left[e_i\right]}=\left[0 \right]$ in $I_n/I_{n+1}$. Since $\left[e_1 \right], \dots, \left[ e_r \right]$ formed a basis for $I_n/I_{n+1}$, we must have $c_i=0$ for all $1 \leq i \leq r$. Hence $\left[ge_1\right], \dots, \left[ge_r \right]$ are linearly independent over $R/\mathfrak{m}_R$, and so $dim_{R/\mathfrak{m}_R}(I_{n+k}/I_{n+k+1}) \geq r$.\\
\end{proof}

\begin{Lemma} \label{semigroup}
Suppose that $\nu$ is an algebraic discrete valuation dominating a local domain $(R,\mathfrak{m}_R)$. There exists $n_0 \in \NN$ such that $n \in S^R(\nu)$ for all $n \geq n_0$.\\
\end{Lemma}

\begin{proof}
Since $\nu$ is a valuation of the quotient field of $R$, $S^R(\nu)=\{\nu(f) \mid f \in R \setminus \{0\}\}$ generates $\Gamma_{\nu}$. To see this, let $\frac{f}{g} \in K^*$, with $f,g \in R \setminus \{0\}$. Then $\nu(\frac{f}{g})=\nu(f)-\nu(g)$. Hence, any element of $\Gamma_{\nu}$ is the difference of two elements in $S^R(\nu)$. In particular, since $\Gamma_{\nu}=\ZZ$, there exist $t,u \in S^R(\nu)$ such that $t-u=1$.\\ 

If $u=0$, then $t=1$, $S^R(\nu)=\NN$ and the lemma follows by taking $n_0=0$. Suppose that $u>0$ and let $i \in \NN$. Then there exist $k,r \in \NN$ such that $i=ku+r$, $0 \leq r < u$, by the division algorithm. Hence $u^2 + i=tr+u(u+k-r) \in S^R(\nu)$. Thus for all $n \geq u^2$, $n \in S^R(\nu)$. Thus the lemma follows by taking $n_0=u^2$.\\
\end{proof}

\begin{Lemma} \label{max}
Suppose that $\nu$ is an algebraic discrete valuation dominating a local domain $(R,\mathfrak{m}_R)$, and that $V_{\nu}/\mathfrak{m}_{\nu}$ is finite over $R/\mathfrak{m}_R$. For all $n \in \NN$, $dim_{R/\mathfrak{m}_R}(I_n/I_{n+1}) \leq |V/\mathfrak{m}_{\nu}:R/\mathfrak{m}_R|$.\\
\end{Lemma}

\begin{proof}
Let $r=|V/\mathfrak{m}_{\nu}:R/\mathfrak{m}_R|$, and let $f_1, \dots f_{r+1}$ be elements in $I_n \setminus I_{n+1}$ such that $[f_1], \dots, [f_{r+1}]$ are distinct elements of $I_n/I_{n+1}$. Then $[1], \dots, \left[\frac{f_{r+1}}{f_1}\right]$ form distinct nonzero residues in $V/\mathfrak{m}_{\nu}$. Hence there exist $c_i \in R/\mathfrak{m}_R$ for $1 \leq i \leq r+1$, not all zero, such that $\sum_{i=1}^{r+1}{c_i\left[\frac{f_i}{f_1}\right]}=[0]$. Thus $\nu(\sum_{i=1}^{r+1}{c_i \frac{f_i}{f_1}})>0$, and so $\nu(\sum_{i=1}^{r+1}{c_i f_i})>n$. Hence $\sum_{i=1}^{r+1}{c_i \left[f_i\right]}=\left[0\right]$, and $[f_1], \dots, [f_{r+1}]$ are linearly dependent over $R/\mathfrak{m}_R$. Thus $dim_{R/\mathfrak{m}_R}(I_n/I_{n+1}) \leq r$.\\ 
\end{proof}

\begin{Proposition} \label{startingnow}
Suppose that $\nu$ is an algebraic discrete valuation dominating a local domain $(R,\mathfrak{m}_R)$, and that $V_{\nu}/\mathfrak{m}_{\nu}$ is finite over $R/\mathfrak{m}_R$. There exists $n_1 \in \NN$ such that for all $n \geq n_1$, $dim_{R/\mathfrak{m}_R}(I_n/I_{n+1})=|V/\mathfrak{m}_{\nu}:R/\mathfrak{m}_R|$.
\end{Proposition}

\begin{proof}
Let $|V/\mathfrak{m}_{\nu}:R/\mathfrak{m}_R|=r$. For $1 \leq i \leq r-1$, there exist elements $\alpha_i \in V \setminus m_{\nu}$, such that the elements $\left[1\right],\left[\alpha_1 \right],\dots,\left[\alpha_{r-1}\right]$ form a basis for $V/\mathfrak{m}_{\nu}$ over $R/\mathfrak{m}_R$. For $1 \leq i \leq r-1$, we may write $\alpha_i=\frac{f_i}{g_i}$, where $f_i, g_i \in R$ and $\nu(f_i)=\nu(g_i)=n_i$. Next, define $g=\prod_{i=1}^{r-1}g_i$. Then $g \in R$, and $\nu(g)=\sum_{i=1}^{r-1}{n_i} =: N$. Further, for all $1 \leq i \leq r-1$, $g  \alpha_i = \frac{gf_i}{g_i} = f_i \cdot \prod_{j \neq i}{g_j}$, so that $g \alpha_i \in R$, and in particular, $\nu(g \alpha_i)=\nu(f_i)+\sum_{j \neq i}{\nu(g_j)}=n_i+\sum_{j \neq i}{n_j}=N$. Thus $g \in I_N \setminus I_{N+1}$ and $g \alpha_i \in I_N \setminus I_{N+1}$ for $1 \leq i \leq r-1$. In particular, $[g]$ and the elements  $[g \alpha_i]$ for $1 \leq i \leq r-1$ are  nonzero elements of $I_N/I_{N+1}$.\\

Suppose that $c_0[g]+\sum_{i=1}^{r-1}{c_i[g\alpha_i]}=[0]$ in $I_N/I_{N+1}$, where $c_i \in R/\mathfrak{m}_R$ for $0 \leq i \leq r-1$. Then $\nu(c_0g+\sum_{i=1}^{r-1}c_ig\alpha_i)>N$, whence $\nu(c_0+\sum_{i=1}^{r-1}c_i\alpha_i)=\nu(c_0g+\sum_{i=1}^{r-1}{c_ig\alpha_i})-\nu(g)>N-N=0$, so that $c_0[1]+\sum_{i=1}^{r-1}c_i[\alpha_i]=[0]$ in $V/\mathfrak{m}_{\nu}$. Hence $c_j=0$ for $0 \leq j \leq r-1$, and so $[g]$ and the elements $[g \alpha_i]$, $1 \leq i \leq r-1$ are linearly independent over $R/\mathfrak{m}_R$ in $I_N/I_{N+1}$. Thus $dim_{R/\mathfrak{m}_R}(I_N/I_{N+1}) \geq r$. By Lemma \ref{max}, $dim_{R/\mathfrak{m}_R}(I_N/I_{N+1})=r$, and so $dim_{R/\mathfrak{m}_R}(I_N/I_{N+1})=|V/\mathfrak{m}_{\nu}:R/\mathfrak{m}_R|$.\\

Finally, suppose that $n_0$ is as in Lemma \ref{semigroup}. Then for all $n \geq n_0$, we have that $n \in S^R(\nu)$. Hence $dim_{R/\mathfrak{m}_R}(I_{N+n}/I_{N+n+1})=|V/\mathfrak{m}_{\nu}:R/\mathfrak{m}_R|$ for all $n \geq n_0$, by Lemma \ref{persist} and Lemma \ref{max}. Thus the lemma follows by taking $n_1=N+n_0$.\\

\end{proof}

Now we prove Theorem \ref{linear} and Corollary \ref{zerolimit} from the introduction. We remind the reader of their statements: \\

 \begin{Theorem}
Suppose that $\nu$ is an algebraic discrete valuation dominating a local domain $(R,\mathfrak{m}_R)$, and that $V_{\nu}/\mathfrak{m}_{\nu}$ is finite over $R/\mathfrak{m}_R$. Then there exist $c,b \in \mathbb{Z}$ such that $$\ell_R(R/I_n)=cn+b$$ for $n\gg 0$, where $c=|V/\mathfrak{m}_{\nu}:R/\mathfrak{m}_R|$.\\
\end{Theorem}

 \begin{Corollary}
Suppose that $\nu$ is an algebraic discrete valuation dominating a local domain $(R,\mathfrak{m}_R)$, and that $V_{\nu}/\mathfrak{m}_{\nu}$ is finite over $R/\mathfrak{m}_R$. Then $$\lim_{n \to \infty} \frac{\ell_R(R/I_n)}{n^2}=0.$$ \\
\end{Corollary}

\begin{proof}
In Proposition \ref{startingnow} we have shown that there exists $n_1 \in \mathbb{N}$ such that $\ell_R(I_n/I_{n+1})=c$ for $n \geq n_1$, where $c=|V/\mathfrak{m}_{\nu}:R/\mathfrak{m}_R|$. Observe that $$\ell_R(R/I_n) = \sum_{k=0}^{k=n-1}{\ell_R(I_k/I_{k+1})}.$$ Hence for $n \geq n_1$, we have $$\ell_R(R/I_n) = \sum_{k=0}^{k=n_1-1}{\ell_R(I_k/I_{k+1})}+\sum_{k=n_1}^{k=n-1}{\ell_R(I_k/I_{k+1})}$$ and hence $$\ell_R(R/I_n) =\sum_{k=0}^{k=n_1-1}{\ell_R(I_k/I_{k+1})}+[n - n_1]c .$$ Thus $$\ell_R(R/I_n) =cn +\left[\sum_{k=0}^{k=n_1-1}{\ell_R(I_k/I_{k+1})}-n_1c \right].$$\\ This proves Theorem \ref{linear}. Dividing both sides by $n^2$ and taking the limit as $n \to \infty$ gives Corollary \ref{zerolimit}.\\ 
\end{proof}

\section{Main Results} \label{infinitesection}

In this section we prove results for algebraic residue field extensions of infinite degree.\\

\begin{Lemma} \label{propassume}
Let $k=L\left( \left\{ \sigma_i \right\} \right)$, where $L$ is a field and $\left\{\sigma_i\right\}_{i \in \mathbb{Z}_+}$ is a set of algebraically independent elements over $L$. Let $\alpha_i = \sqrt{\sigma_i}$ for $i \in \mathbb{Z}_+$. Let $R$ be the power series ring $R:=k[[x,y]]$. Define elements $P_i \in R$ by $$P_0=x, P_1=y, \text{ and } P_{i+1}=P_i^2-\sigma_ix^{2r_i} \text{ for } i \geq 1.$$ Suppose that $r_i \in \mathbb{N}$ are such that $$r_0=1 \text{ and } r_{i+1} > 2r_i \text{ for } i \geq 1.$$ Then there exists a unique valuation $\nu$ of the quotient field of $R$ which dominates $R$, such that $$\nu(P_i)=r_i \text{ for } i \in \mathbb{N}.$$ $\nu$ has the property that $\{P_i\}$ satisfies the conclusions of Theorem 4.2. of \cite{Dale}. The residue field of $\nu$ is naturally isomorphic to $k(\{\alpha_i\})$. We can take $U_i=x^{r_i}$ for $i \geq 1$ in 4) of Theorem 4.2. of \cite{Dale}, and then we obtain that $$\alpha_i=\left[\frac{P_i}{x^{r_i}}\right] $$ for $i \geq 1$.\\ 
\end{Lemma}

\begin{proof}

Let $\beta_i= r_i$ for $i \in \mathbb{N}$. Theorem 1.1 \cite{Dale} and its proof give the existence of a
valuation $\nu$ associated to the $\{\beta_i\}_{i \in \mathbb{N}}$ and $\{\alpha_i\}_{i \in \mathbb{Z}_+}$, which dominates $R$. The proof (from the middle of page 21 of \cite{Dale}) determines $\nu$ by constructing a sequence of polynomials which
are a generating sequence determining the valuation $\nu$. Now the sequence of polynomials
are in fact the generating sequence associated to $\nu$ by the algorithm of Theorem 4.2 \cite{Dale}. If
we start the generating sequence with $x$ and $y$, and we take the $U_i$ in the construction to
be $x^{r_i}$ (which is consistent with the construction), we have that the generating sequence
constructed in the proof of Theorem 1.1 is in fact the $\{P_i\}$, and the algorithm of Theorem
4.2 produces the generating sequence $\{P_i\}$.

\end{proof}

\begin{Remark} \label{necessary}
 In the choice of the sequence $\{r_i\}_{i \in \mathbb{N}}$ determining the valuation $\nu$ in Lemma \ref{propassume}, the condition that $r_{i+1}>2r_i$ for all $i \geq 1$ is necessary for the conclusions of Lemma \ref{propassume}, by the use of the algorithm of \cite{Dale}.  
\end{Remark}

\begin{Lemma} \label{basis}
A $k$-basis of $I_s/I_{s+1}$ is \begin{multline} $$\mathcal{B}_s=\{[P_0^{n_0}P_1^{n_1}\cdots P_i^{n_i}] \mid i \in \mathbb{N}, n_0,\ldots, n_i\in \mathbb{N}, \\
n_j \in \{0,1\} \text{ for } 1 \leq j \leq i, \text{ and } n_0+n_1r_1+\cdots + n_ir_i = s\}.$$ \end{multline} \\
\end{Lemma}

\begin{proof}
We first show that the distinct elements of $\mathcal{B}_s$ are linearly independent over $k$. Let us fix an enumeration $\{\gamma_h\}$ of the monomials \begin{multline} $$\{P_0^{n_0}P_1^{n_1}\cdots P_i^{n_i} \mid i \in \mathbb{N}, n_0,\ldots, n_i\in \mathbb{N}, \\
n_j \in \{0,1\} \text{ for } 1 \leq j \leq i, \text{ and } n_0+n_1r_1+\cdots + n_ir_i = s\}.$$\end{multline} Observe that $\{[\gamma_h]\}$ is an enumeration of the elements of $\mathcal{B}_s$. By construction, $\nu(\gamma_1)=\nu(\gamma_h)=s$ for all $h$. Furthermore, the exponent $i$-tuples $(n_0,  \dots, n_i)$ associated to the monomials $\gamma_h$ are distinct for distinct $h$.  \\

Suppose now that there is an equation $\sum c_h[\gamma_h]  =0$ in $I_s/I_{s+1}$, with $c_h \in k$. Then $\nu(\sum c_h\gamma_h)>s$, and so $\nu(\sum c_h\frac{\gamma_h}{\gamma_1})=\nu(\sum c_h\gamma_h)-\nu(\gamma_1)>s-s=0$. Thus we have an equation $\sum c_h[\frac{\gamma_h}{\gamma_1}]=0$ in $V_{\nu}/\mathfrak{m}_{\nu}$. Now Theorem 4.2 (2) of \cite{Dale} states that the $\{[\frac{\gamma_h}{\gamma_1}]\}$ are linearly independent over $k$. Thus we must have had $c_h=0$ for all $h$, and so the elements of $\mathcal{B}_s$ were linearly independent over $k$.\\

Next, by reducing equation (36) modulo $I_{s+1}$ in the proof of Theorem of 1.1 in \cite{Dale}, we see that the elements of $\mathcal{B}_s$ generate $I_s/I_{s+1}$ over $k$.\\ 

Whence $\mathcal{B}_s$ is a $k$-basis for $I_s/I_{s+1}$. \\
\end{proof}


\begin{Definition} \label{sequence}
 Let $\theta \in \mathbb{R}_{\geq 0} \cup \{\infty \}$. 
 
 \begin{enumerate}[1.]
  \item If $\theta < \infty$, then define an associated sequence $\{e_j\}_{j \geq 2}$ as follows: define $e_2 \in \mathbb{N}$ so that $\frac{e_2}{2} \leq \theta < \frac{e_2+1}{2}$, and define for every $j \geq 3$, $e_j \in \{0,1\}$ such that $0 \leq \theta-e_2\frac{1}{2} - \dots - e_j\frac{1}{2^{j-1}} < \frac{1}{2^{j-1}}.$ Further, for every $j \geq 2$, define $R_j:= \theta - \sum_{k=2}^{j}{e_k2^{1-k}}$.\\
  \item If $\theta = \infty$, then define an associated sequence $\{e_j\}_{j \geq 2}$ by $e_j = 2^{j-1}$ for every $j \geq 2$. Notice in this case that $\sum_{j=2}^{i}{e_j2^{1-j}} = i-1$.  \\
 \end{enumerate}
 
\end{Definition}

In the sequel, we will have the following assumptions. We assume $\theta \in \mathbb{R}_{\geq 0} \cup \{ \infty \}$ given, and we define the sequence $\{r_i\}_{i \geq 0}$ inductively by \\

\begin{equation} \label{inductive_defn}
r_0=1, r_1=1, \text{ and }  r_{i+1}=2r_i+1+e_{i+1} \text{ for all } i \geq 1.\\  
\end{equation}\\

Then the general term $r_i$ is given by the formula

\begin{equation} \label{generalformula} 
r_0 = 1 \text{ and } 
r_i=2^{i}-1+\sum_{k=2}^{i}{e_k2^{i-k}}, 
\end{equation}

for $i \geq 1$, where $\{e_j\}$ is as in Definition \ref{sequence}. 

\begin{Lemma} \label{easy}

With our assumption (\ref{inductive_defn}), we have that

\begin{enumerate}[1.]
 \item $r_{i+1} > 2r_i$ for all $i \geq 1$,  
 \item $2r_i > r_0 + \dots + r_i$ for all $i \geq 2$,  
\end{enumerate}

\end{Lemma}

\begin{proof} 

Since $r_{i+1}=2r_i+1+e_{i+1} > 2r_i \text{ for all } i \geq 1$ by construction, the first claim follows.\\

We verify the second claim by induction. By (\ref{inductive_defn}), $r_2 = 3+e_2 > 2 = r_0 + r_1$, and the second claim holds for the case $i=2$. \\

Now suppose that $i>2$, and that the second claim holds for the case $k=i-1$. We have that $2r_i=r_i+r_i>r_i+(r_0 + \dots + r_{i-1})$, the last inequality by the inductive statement. Hence the second claim holds for all $i \geq 2$.\\



\end{proof}

By the first claim of Lemma \ref{easy}, the sequence $\{r_i \}$ defined by (\ref{inductive_defn}) satisfies the assumptions of Lemma \ref{propassume}. Let $\nu$ be the valuation defined by Lemma \ref{propassume} with the sequence $\{r_i \}$ defined by (\ref{inductive_defn}).

\begin{Remark} \label{spaceremark}
In the sequel, we will have inductively chosen the sequence $\{r_i\}_{i \in \mathbb{Z}_+}$ such that $$r_{i+1} \geq \max(2r_{i},r_0+r_1+ \dots +r_{i})+1.$$ \\
\end{Remark}


\begin{Lemma} \label{slab}
For any $i \in \mathbb{Z}_{\geq 1}$, let $s \in (r_0+r_1+ \dots +r_i , r_{i+1})$. Then $dim_k(I_s/I_{s+1})=2^i.$\\
\end{Lemma}

\begin{proof} 
Define the following function: $$N_0(i,s, n_1, \dots, n_i) = s-(n_1r_1+\dots + n_ir_i), $$ where $n_1, \dots, n_i \in \{0,1\}$, as in Lemma \ref{basis}. Since $s>r_1+\dots + r_i \geq n_1r_1+\dots + n_ir_i$, $N_0(i,s,n_1, \dots, n_i)$ is nonnegative. Note that for any choice of $n_1, \dots, n_i$, we have that $N_0(i,s,n_1, \dots, n_i)+n_1r_1+\dots +n_ir_i=s$, so that $[P_0^{N_0(i,s,n_1, \dots, n_i)}P_1^{n_1}\cdots P_i^{n_i}] \in \mathcal{B}_s$.\\

Furthermore, since $r_j>s$ for $j>i$ by Lemma \ref{easy}, we have that $n_0+n_1r_1+\cdots + n_{j-1}r_{j-1} + n_{j}r_{j} \neq s$ for \emph{any} choice of $j \geq i+1$ and $n_0, n_1, \dots, n_{j-1}, n_{j} \in \mathbb{N}$ with $n_j > 0$.\\

Thus the elements of $\mathcal{B}_s$ are precisely the following:\\

$$\{[P_0^{N_0(i,s,n_1, \dots, n_i)}P_1^{n_1}\cdots P_i^{n_i}] \mid n_1, \dots, n_i \in \{0,1\} \} .$$\\

Whence $dim_k(I_s/I_{s+1})=2^i.$\\

\end{proof}




\begin{Lemma} \label{othereasy}
Let $\{r_i\}_{i \geq 0} \subset \mathbb{Z}_{\geq 0}$ be a sequence, let $ \max(2r_2,r_0+r_1+r_2)=2r_2>r_0+r_1+r_2$, and let $r_{i+1}>\max(2r_i,r_0+\dots +r_i)$ for $i \geq 2$. Then $\max(2r_i,r_0+\dots +r_i)=2r_i>r_0+\dots r_i$ for all $i \geq 2$.
\end{Lemma}

\begin{proof} 

The case $i=2$ is given by hypothesis. Now suppose that the statement holds for some $k \geq 2$. Then $2r_{k+1} = r_{k+1} + r_{k+1} > r_{k+1} + \max(2r_k,r_0 + \dots + r_k) = r_{k+1} + 2r_k > r_{k+1} + (r_0 + \dots + r_k)$. Hence $\max(2r_{k+1},r_0 + \dots + r_{k+1}) = 2r_{k+1} > r_0 + \dots + r_{k+1}$, and the result follows by induction. \\
\end{proof}

\begin{Remark} \label{check}
 Suppose $\{r_i\}_{i \in \mathbb{Z}_{\geq 0}}$ is constructed as in (\ref{generalformula}). Then $r_{i+1}>\max(2r_i,r_0+\dots+ r_i)$ for all $i \geq 2$ and the conclusions of Lemma \ref{othereasy} hold. By Remark \ref{spaceremark} and Lemma \ref{othereasy} it is enough to show that $2r_2 = 6 + 2e_2 > 5+e_2 = r_0 + r_1 + r_2$.  \\

\end{Remark}

\begin{Lemma} \label{growth}
 For all $i \in \mathbb{Z}_{>0}$ we have that $r_{i+1}-(r_0+\dots + r_i) \geq i$.
\end{Lemma}

\begin{proof}
 From (\ref{generalformula}) we can deduce that $r_{i+1}=2^{i+1}-1+\sum_{j=2}^{i+1}{e_j2^{i+1-j}}$ and $\sum_{j=0}^{i}{r_j} = 2^{i+1} - 1 - i + \sum_{j=2}^{i}\sum_{k=2}^{j}e_k2^{j-k}$. Hence $$r_{i+1} - \sum_{j=0}^{i}r_j = i + \sum_{k=2}^{i+1}e_k2^{i+1-k} - \sum_{j=2}^{i}\sum_{k=2}^{j}e_k2^{j-k}$$ and it is enough to show that $\sum_{k=2}^{i+1}e_k2^{i+1-k} - \sum_{j=2}^{i}\sum_{k=2}^{j}e_k2^{j-k} \geq 0$. \\

 
 Suppose first that $\theta < \infty$. We have that $$\sum_{k=2}^{i+1}{e_k2^{i+1-k}} - \sum_{j=2}^{i}{\sum_{k=2}^{j}{e_k2^{j-k}}}$$ 

 $$= e_{i+1} + 2^{i}\sum_{k=2}^{i}{e_k2^{1-k}} - \sum_{j=2}^{i}{2^{j-1}\sum_{k=2}^{j}{e_k2^{1-k}}}$$ 

 $$= e_{i+1} + 2^{i}(\theta-R_i) - \sum_{j=2}^{i}{2^{j-1}(\theta - R_j)}$$
 
 $$\geq  e_{i+1} + 2^{i}(\theta-R_i) - \sum_{j=2}^{i}{2^{j-1}(\theta - R_i)}$$
 
 $$=  e_{i+1} + 2^{i}(\theta-R_i) - (2^{i}-2)(\theta - R_i)$$

 $$=  e_{i+1} + 2(\theta-R_i) \geq 0,$$ \\

 the penultimate inequality since for all $1 \leq j \leq i$, we have that $\theta - R_j \leq \theta - R_i$. \\
 
 Suppose next that $\theta = \infty$. We have that $$\sum_{k=2}^{i+1}{e_k2^{i+1-k}} - \sum_{j=2}^{i}{\sum_{k=2}^{j}{e_k2^{j-k}}}$$ 

 $$= e_{i+1} + 2^{i}\sum_{k=2}^{i}{e_k2^{1-k}} - \sum_{j=2}^{i}{2^{j-1}\sum_{k=2}^{j}{e_k2^{1-k}}}$$ 

 $$= e_{i+1} + 2^{i}(i-1) - \sum_{j=2}^{i}{2^{j-1}(j-1)}$$
 
 $$\geq  e_{i+1} + 2^{i}(i-1) - \sum_{j=2}^{i}{2^{j-1}(i-1)}$$
 
 $$=  e_{i+1} + 2^{i}(i-1) - (2^{i}-2)(i-1)$$

 $$=  e_{i+1} + 2(i-1) \geq 0.$$

\end{proof}

\begin{Theorem} \label{example}
There exists a regular local ring $(R,\mathfrak{m}_R)$ of dimension 2 and a discrete, rank 1 valuation $\nu$ of the quotient field of $R$ dominating $R$, such that the function $\alpha(n)=\ell_R(I_n/I_{n+1})$ is not a quasi-polynomial plus a bounded function for large integers $n$.  
\end{Theorem}

\begin{proof}

Take $R=k[[x,y]]$ and consider the valuation $\nu$ constructed by the generating sequence of Lemma \ref{propassume}, along with the associated function $\alpha(n)=\ell_R(I_n/I_{n+1})$. We will show that if $\nu$ has a generating sequence defined inductively by: $\nu(P_0)=r_0=1$, $\nu(P_1)=r_1=1$ and $\nu(P_{i+1})=r_{i+1}$ where $\{r_i \}$ is defined by (\ref{inductive_defn}), then $\alpha(n)$ cannot be the sum of a quasipolynomial and a bounded function for large $n$. \\

Suppose by way of contradiction that there exists $n_0>0$ such that $$\alpha(n)=A(n)+b(n),$$ for $n \geq n_0$, where $$A(n)=\alpha_d(n)n^d+\dots + \alpha_1(n)n+\alpha_0(n),$$ the $\alpha_i(n)$ are periodic functions of integral period, $s$ is their common integral period, and $-M \leq b(n) \leq M$ is a bounded function. By defining $b(n)=\alpha(n)-A(n)$ for $n\le n_0$, we can take $n_0$ to be $1$. \\

For each $0 \leq p <s$, consider the function $$A_p(t):= A(p+st) =  \alpha_d(p+st)(p+st)^d + \dots + \alpha_1(p+st)(p+st)+\alpha_0(p+st).$$ The coefficients $\alpha_d(p+st), \dots, \alpha_1(p+st), \alpha_0(p+st)$ are constant functions of $t \in \mathbb{Z}_+$, and are each equal to $\alpha_d(p), \dots, \alpha_1(p), \alpha_0(p)$, respectively. Expanding the powers of $p+st$ shows that these are polynomials of degree $ \leq d$ in $t$. 
Thus we may consider $s$ (not necessarily distinct) polynomial functions of $t \in \mathbb{Z}_+$:  $$A_p(t)=\alpha_d(p)(p+st)^d + \dots + \alpha_1(p)(p+st)+\alpha_0(p),$$ and also $s$ (not necessarily distinct) bounded functions of $t \in \mathbb{Z}_+$:  $$b_p(t)=b(p+st),$$ for $0 \leq p < s$.\\ 

For each $0 \leq p <s$, define $f_p(t) := A_p(t)+b_p(t)$. For any function $\psi(t)$, define $\Delta_t \psi(t) := \psi(t) - \psi(t-1)$. Thus for all $0 \leq p <s$, $\Delta_t f_p(t) = \Delta_t A_p(t) + \Delta_t b_p(t)= A_{p}(t) - A_{p}(t-1) +b_{p}(t) - b_{p}(t-1)$.\\

The sequence $\{ r_i \}_{i \in \mathbb{Z}_+}$ satisfies $r_{i+1}-(r_0+\dots+r_i) \geq i$. Hence $r_{i+1}>r_0+\dots+ r_i +2s$ for all $i$ sufficiently large. Then by Lemma \ref{slab}, for each $0 \leq p < s$, $\Delta_t f_p(t)$ has infinitely many zeroes in $\mathbb{N}$. Therefore for each $0 \leq p < s$, there is an infinite increasing sequence $\{t_l\}_{l \in \mathbb{N}} \subset \mathbb{N}$ such that such that $r_0+\dots +r_l<p+(t_l-1)s<p+t_ls<r_{l+1}$ and $\Delta_t A_p(t_l) +\Delta_t b_p(t_l) =0$; i.e.: $\Delta_t A_p(t_l) =-\Delta_t b_p(t_l)$. Hence the sequence $\{\Delta_t A_p(t_l)\}_{l \in \mathbb{N}}$ is bounded (by $-2M$ and $2M$).\\

For each $0 \leq p < s$, $\Delta_t A_p(t)$ is a polynomial, and is a continuous function on $\mathbb{R}$. If $\Delta_t A_p(t)$ is nonconstant, then either $\lim_{l \rightarrow \infty}{\Delta_t A_p(t_l)}=\infty$ or $\lim_{l \rightarrow \infty}{\Delta_t A_p(t_l)}=-\infty$. Thus for each $0 \leq p < s$, $\Delta_t A_p(t)$ is constant, and therefore $A_p(t)$ is either constant or linear. For $0 \leq p < s$, write $A_p(t)=m_pt+c_p$. \\

We cannot have $m_p \leq 0$ for any $0 \leq p <s$. By way of contradiction, suppose that $m_{p'} \leq 0$. Recall that $b_{p'}(t)=f_{p'}(t) - m_{p'}t-c_{p'}$. By Lemma \ref{slab}, and since $r_{i+1}>r_0+\dots +r_i +s$ for all $i$ sufficiently large, we can find an increasing sequence $\{t_k\}_{k \in \mathbb{N}}$ such that $r_0+\dots +r_k<p'+t_ks<r_{k+1}$ and $f_{p'}(t_k)=2^k$ hold for all $k\gg 0$. Then for all $k\gg 0$, we have $b_{p'}(t_k)=2^k- m_{p'}t_k+c_{p'}$, and so letting $k \rightarrow \infty$, we get that $b_{p'}(t_k) \rightarrow \infty$, contradicting that $b_{p'}(t)$ was bounded. Thus, all the $m_p > 0$ for $0 \leq p < s$.\\

Let $0 \leq p <s$ be given. Then $b_p(t)=f_p(t)-m_pt-c_p$ for $t \in \mathbb{Z}_+$. Consider the subsequence $\{i_{s,k}\}_{k \in \mathbb{N}}$ such that $i_{s,k}= s(k+1)+1$. Then by Lemma \ref{growth}, $r_{i_{s,k}+1} - (r_0+\dots+r_{i_{s,k}}) > s(k+1)$. For each $k \in \mathbb{N}$, define $t_{k,max}=\max\{t \in \mathbb{Z}_+ \mid p+st < r_{i_{s,k}+1} \}$ and $t_{k,min}=\min\{t \in \mathbb{Z}_+ \mid p+st > r_0 + \dots +r_{i_{s,k}} \}$. Then $f_p(t)=2^{i_{s,k}}$ for $t_{k,min} \leq t \leq t_{k,max}$ by Lemma \ref{slab}, and $t_{k,max} - t_{k,min} \geq k$. Thus for all $k \in \mathbb{N}$, $$b_p(t_{k,max})-b_p(t_{k,min})=m_p(t_{k,min}-t_{k,max}) \leq -km_p.$$ Let $N \in \mathbb{N}$ be arbitrary and let $k>\frac{N}{m_p}$. Then $b_p(t_{k,max})-b_p(t_{k,min})<-N$. Now taking $N>2M$ contradicts that $-M \leq b(n) \leq M$ was bounded, since if $-M \leq b(n) \leq M$, then $-2M \leq b(n_1)-b(n_0) \leq 2M$ for any $n_0,n_1 \in \mathbb{Z}_+$. \\

Thus we see that for this choice of the sequence $\{r_i \}_{i \in \mathbb{Z}_+}$, $\alpha(n)$ cannot be written as a quasi-polynomial plus a bounded function. This proves the theorem.\\

\end{proof} 

\begin{Corollary} \label{secondnon-poly}
There exists a regular local ring $(R,\mathfrak{m}_R)$ of dimension 2 and a discrete, rank 1 valuation $\nu$ of the quotient field of $R$ dominating $R$, such that the function $\ell_R(R/I_n)$ cannot be written as a quasi-polynomial plus a bounded function for large integers $n$.  
\end{Corollary}

\begin{proof}
Let $R$ and $\nu$ be as in the proof of Theorem \ref{example}, and suppose that $\ell_R(R/I_n) = Q(n)+\sigma(n)$ for $n \gg 0$, where $Q(n)=\sum_{l=0}^{l=d}{\alpha_l(n)n^l}$ is a quasi-polynomial and $\sigma(n)$ is a bounded function, say $|\sigma(n)| \leq M$. Then we have that: $$\ell_R(I_n/I_{n+1}) = \ell_R(R/I_{n+1}) - \ell_R(R/I_n) = (Q(n+1)-Q(n)) + (\sigma(n+1) - \sigma(n))$$ Let $s$ be the common integral period of the coefficients $\{\alpha_l(n)\}_{0 \leq l \leq d}$ of $Q(n)$. Then $$Q(n+1)-Q(n)=\sum_{p=0}^{p=d}\left({\left[\sum_{l=p}^{l=d}{\choos{l}{p}\alpha_l(n+1)}\right]-\alpha_p(n)}\right)n^p $$ is also a quasi-polynomial with coefficients of  integral period $s$. Further, $-2M \leq \sigma(n+1) - \sigma(n) \leq 2M$ is a bounded function. Thus $\ell_R(I_n/I_{n+1})$ is a quasi-polynomial plus a bounded function, contradicting Theorem \ref{example}.\\
\end{proof}

\begin{Proposition} \label{recursive}
 Let $\alpha(n)=\ell_R(I_n/I_{n+1})$. Then the following recursive relation holds: for any $n \in \mathbb{Z}_{> 0}$, consider the unique $i \in \mathbb{Z}_{\geq 0}$ such that $n \in [r_i,r_{i+1})$. Then we have that $\alpha(n)=\alpha(r_i-1)+\min\{\alpha(r_i-1),\alpha(n-r_i)\}.$
\end{Proposition}


\begin{proof}
 Let $n \in \mathbb{Z}_{\geq 0}.$ Recall from Lemma \ref{basis} that a $k$-basis for $I_n/I_{n+1}$ is given by the set $$\mathcal{B}_n=\{[P_0^{n_0}P_1^{n_1}\dots P_i^{n_i}] \mid i \geq 0, n_k \in \mathbb{Z}_{\geq 0} \text{ for } k \geq 0, n_k \in \{0,1\} \text{ for } 1 \leq k \leq i, n_0+\sum_{k=1}^{i}{n_kr_k}=n\}.$$ Thus we may canonically identify the set $\mathcal{B}_n$ with the set of tuples 
 
 \begin{equation} \label{tuples}
\mathcal{B}_n \cong \{(n_0,n_1,\dots,n_i) \in \mathbb{N} \times \{0,1\}^i \mid n_0 + \sum_{k=1}^{i}{n_kr_k}=n \}.  
 \end{equation}

 where $i$ is the unique $j \in \mathbb{Z}_{\geq 0}$ such that $n \in [r_j,r_{j+1})$.\\
 
 Define $\mathcal{B}_n^0 = \{(n_0,n_1,\dots,n_i) \in \mathcal{B}_n \mid n_i=0 \}$ and $\mathcal{B}_n^1 = \{(n_0,n_1,\dots,n_i) \in \mathcal{B}_n \mid n_i=1 \}$.\\
 
 Let us first show that $|\mathcal{B}_m^0| = |\mathcal{B}_{r_i-1}|$. Suppose that $(n_0,n_1,\dots,n_i) \in \mathcal{B}_n^0$. Then $n_0+\sum_{j=1}^{i-1}{n_jr_j}=n$. By Lemma \ref{easy}, $\sum_{j=1}^{i-1}{n_jr_j} \leq \sum_{j=1}^{i-1}{r_j} < r_i \leq n$, hence $n_0=n-\sum_{j=1}^{i}{n_jr_j} > n-r_i \geq 0$ and so $n_0-n+r_i-1 \geq 0$. Let us define a mapping $\lambda : \mathcal{B}_n^0 \rightarrow \mathcal{B}_{r_i-1}$ by $\lambda((n_0,n_1,\dots,n_{i-1},n_{i}))=(n_0-n+r_i-1, n_1, \dots, n_{i-1})$. Notice that $\lambda$ is injective. For, suppose that $\lambda((n_0,n_1,\dots,n_{i-1},n_{i}))=(n_0-n+r_i-1, n_1, \dots, n_{i-1})=(n_0'-n+r_i-1, n_1', \dots, n_{i-1}')=\lambda((n_0',n_1',\dots,n_{i-1}',n_{i}'))$. By (\ref{tuples}) for the case $m=r_i-1$ we have that $n_j=n_j'$ for $1 \leq j \leq i$, whence $n_0=n_0'$ and $\lambda$ is injective. Furthermore, if $(n_0,n_1,\dots,n_{i-1}) \in \mathcal{B}_{r_i-1}$, then $(n_0+n-r_i+1,n_1,\dots,n_{i-1},0) \in \mathcal{B}_n$ and $\lambda((n_0+n-r_i+1,n_1,\dots,n_{i-1},0))=(n_0,n_1,\dots,n_{i-1})$, whence $\lambda$ is surjective. Thus $\lambda$ gives a bijection between the finite sets $\mathcal{B}_n^0$ and $\mathcal{B}_{r_i-1}$, and $|\mathcal{B}_n^0| = |\mathcal{B}_{r_i-1}|$. \\
 
 Next, we shall show that $$
|\mathcal{B}_n^1| =  \left\{
        \begin{array}{ll}
            |\mathcal{B}_{r_i-1}| & \text{ if } n-r_i \geq r_i \\
            |\mathcal{B}_{n-r_i}| &  \text{ if } n-r_i < r_i
        \end{array}
    \right.
$$

To this end, suppose that $(n_0,n_1,\dots,n_{i-1},n_i=1) \in \mathcal{B}_n^1$. Then $n_0+\sum_{j=1}^{i-1}{r_jn_j}=n-r_i$. So, we have a (injective) mapping $\mu : \mathcal{B}_n^1 \rightarrow \mathcal{B}_{n-r_i}$, given by $\mu((n_0,n_1,\dots,n_{i-1},n_i=1))=(n_0,n_1,\dots,n_{i-1})$. We claim that if $n-r_i<r_i$, then $\mu$ is a bijection between $\mathcal{B}_n^1$ and $\mathcal{B}_{n-r_i}$, and if $n-r_i \geq r_i$, then $\mu$ is a bijection between $\mathcal{B}_n^1$ and $\mathcal{B}_{n-r_i}^0$. \\

Suppose that $n-r_i<r_i$. Then in any representation $n-r_i=n_0+\sum_{j=1}^{i}{r_jn_j}$ as in (\ref{tuples}), we must have $n_i=0$. It follows that $\mu$ is onto. Similarly, if $n-r_i \geq r_i$, we have that $\mu$ is onto $\mathcal{B}_{n-r_i}^0$ since $n-r_i < r_{i+1}$. Further, $\mathcal{B}_{n-r_i}^0$ maps bijectively onto $\mathcal{B}_{r_i-1}$ via the map $(n_0,n_1,\dots,n_{i-1},n_i=0) \rightarrow (n_0-n+2r_i-1,n_1,\dots,n_{i-1})$.\\

Notice next that by Lemma \ref{persist}, if $n-r_i \geq r_i$, then $|\mathcal{B}_{n-r_i}| \geq |\mathcal{B}_{r_i-1}|$, and if $n-r_i < r_i$, then $|\mathcal{B}_{n-r_i}| \leq |\mathcal{B}_{r_i-1}|$. Thus we obtain that $|\mathcal{B}_n^1|=\min\{|\mathcal{B}_{n-r_i}|, |\mathcal{B}_{r_i-1}|\}$.\\

Hence the formula $\alpha(n)=\alpha(r_i-1)+\min\{\alpha(n-r_i),\alpha(r_i-1)\}$ holds.

\end{proof}


\begin{Lemma} \label{epsilon}
For all $i \geq 2$, the following holds:\\
 
\begin{enumerate}[1.]
 \item If $\theta < \infty$, then $\frac{\alpha(r_i-1)}{r_i}=\frac{1}{2+\theta}+ \varepsilon_i$, where $0 < \varepsilon_i := \frac{R_i+2^{1-i}}{(2+\theta)[2+\theta-R_i-2^{1-i}]}$. (Recall that for every $i \geq 1$, $R_i:=\theta-\sum_{j=1}^{i}{e_j2^{-j}}$.)
 \item If $\theta = \infty$, then $\frac{\alpha(r_i-1)}{r_i}=\varepsilon_i$, where $0 < \varepsilon_i := \frac{1}{1+i-2^{1-i}}$. 
\end{enumerate}
\end{Lemma}

\begin{proof}
 Suppose first that $\theta < \infty$. Observe that $$\frac{\alpha(r_i-1)}{r_i}=\frac{2^{i-1}}{2^{i}-1+\sum_{j=2}^{i}{e_j2^{i-j}}},$$ by Lemma \ref{slab}.\\
 
 Hence, $$\frac{\alpha(r_i-1)}{r_i} = \frac{1}{2-\frac{1}{2^{i-1}}+\sum_{j=2}^{i}{e_j2^{1-j}}},$$
 
 $$= \frac{1}{2-\frac{1}{2^{i-1}}+(\theta-R_i)},$$

 $$= \frac{1}{(2+\theta)(1-\frac{1}{(2+\theta)2^{i-1}}-\frac{R_i}{2+\theta})},$$
 
 $$= \frac{(1-\frac{1}{(2+\theta)2^{i-1}}-\frac{R_i}{2+\theta}) + (\frac{1}{(2+\theta)2^{i-1}}+\frac{R_i}{2+\theta})}{(2+\theta)(1-\frac{1}{(2+\theta)2^{i-1}}-\frac{R_i}{2+\theta})},$$
 
 $$= \frac{1}{2+\theta} + \varepsilon_i.$$
 
 Lastly, observe that for each $i \geq 2$, $0 \leq \theta - R_i$. Hence $-2^{1-i} \leq \theta -R_i - 2^{1-i}$, and $2-2^{1-i} \leq 2+ \theta -R_i - 2^{1-i}$. Hence $2+ \theta -R_i - 2^{1-i} > 0$. Further, $R_i \geq 0$, so $\varepsilon_i > 0$ follows.\\
 
  Next suppose that $\theta = \infty$. Observe that $$\frac{\alpha(r_i-1)}{r_i}=\frac{2^{i-1}}{2^{i}-1+\sum_{j=2}^{i}{e_j2^{i-j}}},$$ by Lemma \ref{slab}.\\
 
 Hence, $$\frac{\alpha(r_i-1)}{r_i} = \frac{1}{2-\frac{1}{2^{i-1}}+\sum_{j=2}^{i}{e_j2^{1-j}}},$$
 
 $$= \frac{1}{2-\frac{1}{2^{i-1}}+(i-1)},$$
 
 $$= \frac{1}{1+i-\frac{1}{2^{i-1}}},$$

 $$= \varepsilon_i.$$\\
 
 Observe that $\varepsilon_2=\frac{2}{5}>0$. Suppose that $i \geq 3$ is given, and $\varepsilon_{i-1}>0$. We have that $\frac{1}{\varepsilon_i}=\frac{1}{\varepsilon_{i-1}}+(1+\frac{1}{2^{i-1}})>\frac{1}{\varepsilon_{i-1}}>0$. Hence $\varepsilon_i > 0$ for $i \in \mathbb{N}$ by induction.\\
 
\end{proof}

\begin{Corollary} \label{limit}

\begin{enumerate}[1.]
 \item Suppose that $\theta < \infty$. Then for all $i \geq 2$, $\frac{\alpha(r_i-1)}{r_i}>\frac{1}{2+\theta}$. Furthermore, $\lim_{i \rightarrow \infty} \varepsilon_i = 0$, and so $\lim_{i \rightarrow \infty} \frac{\alpha(r_i-1)}{r_i} = \frac{1}{2+\theta}$.
 \item Suppose that $\theta = \infty$. Then for all $i \geq 2$, $\frac{\alpha(r_i-1)}{r_i}>0$. Furthermore, $\lim_{i \rightarrow \infty} \varepsilon_i = 0$, and so $\lim_{i \rightarrow \infty} \frac{\alpha(r_i-1)}{r_i} = 0$.
 \end{enumerate}

\end{Corollary}

\begin{proof}
 Suppose that $\theta<\infty$. Then the first claim follows since $\varepsilon_i>0$ and the second claim follows since $0\leq R_i<2^{1-i}$.\\
 
 Next, suppose that $\theta=\infty$. Then the first claim follows since $\varepsilon_i>0$ and the second claim follows since $1+i-\frac{1}{2^{i-1}} > i$.\\
 
\end{proof}

\begin{Lemma} \label{bigger}
\begin{enumerate}[1.]
 \item Suppose that $\theta < \infty$. Then we have that $\alpha(n) > (\frac{1}{2+\theta})n$ for all $n \geq 0$.\\
 \item Suppose that $\theta = \infty$. Then we have that $\frac{\alpha(n)}{n} > 0$ for all $n \geq 1$.\\
\end{enumerate}

\end{Lemma}

\begin{proof}
 
 First suppose that $\theta < \infty$. \\

Observe first that if $n=0$, then $\alpha(0)=1 > (\frac{1}{2+\theta})0$, and if $n=1$, then $\alpha(1)=2>1>\frac{1}{2+\theta}(1)$. If $n \in [r_0+r_1,r_2)=[2,r_2)$, then $\frac{\alpha(n)}{n} = \frac{2}{n} \geq \frac{2}{r_2} = \frac{2}{3+e_2}$. By construction, $e_2 \leq 2\theta$, so $\frac{\alpha(n)}{n} \geq \frac{2}{3+e_2} \geq \frac{2}{3+2\theta} > \frac{2}{4+2\theta} = \frac{1}{2+\theta}$. Hence the result holds for all $n \in [0,r_2)$.\\
 
 
Let $n \geq 2$ be given, and consider the unique $i \in \mathbb{Z}_{\geq 1}$ such that $n \in [r_i,r_{i+1})$. We will prove the result by induction on $i$. The case $i=1$ has already been verified. Assume then that $i \geq 2$. We will use the formula 

 
  $$\alpha(n) = \alpha(r_i-1) + \min\{ \alpha(r_i-1), \alpha(n-r_i) \},$$\\

 and the observation that if \\
 
\begin{equation} \label{first_fraction}
 \alpha(r_i-1) > \left(\frac{1}{2+\theta}\right)r_i 
\end{equation}
 
 and 
 
 \begin{equation} \label{second_fraction}
\min\{\alpha(r_i-1), \alpha(n-r_i) \} \geq \left(\frac{1}{2+\theta} \right)(n-r_i),  
 \end{equation}

 then upon adding these two inequalities, we obtain
 
 $$\alpha(n) > \left( \frac{1}{2+\theta} \right)n.$$ \\

 Observe that we already have (\ref{first_fraction}) from Corollary \ref{limit}.\\
 
  We now check (\ref{second_fraction}). Assume the inductive statement holds: that $r_i \leq n < r_{i+1}$, $i \geq 2$, and that $\alpha(m)> \left( \frac{1}{2+\theta}\right) m $ for all $0 
  \leq m < r_i$.\\
  
    If $n-r_i <r_i$, then by Lemma \ref{persist}, $$\min\{\alpha(r_i-1), \alpha(n-r_i) \}=\alpha(n-r_i) > \left(\frac{1}{2+\theta}\right)(n-r_i),$$ so (\ref{second_fraction}) holds and we are done. \\
 
 Suppose then that $n-r_i \geq r_i$, in which case $\min\{\alpha(r_i-1), \alpha(n-r_i) \}=\alpha(r_i-1)$ by Lemma \ref{persist}. Observe that since $r_i \leq n < r_{i+1}$, we have $0 \leq n-r_i \leq r_{i+1}-1-r_i$, so that $\frac{\alpha(r_i-1)}{n-r_i} \geq \frac{\alpha(r_i-1)}{r_{i+1}-1-r_i}$. We have $$r_{i+1}-1-r_i = 2^{i}-1+\sum_{j=2}^{i}{e_j2^{i-j}} + e_{i+1},$$ so $$\frac{\alpha(r_i-1)}{n-r_i} \geq \frac{2^{i-1}}{2^{i}-1+\sum_{j=2}^{i}{e_j2^{i-j}} + e_{i+1}}. $$ It follows that $$\frac{\alpha(r_i-1)}{n-r_i} \geq \frac{1}{2+\theta -R_i + \frac{(e_{i+1}-1)}{2^{i-1}}} $$ and it suffices to show that $-R_i+\frac{e_{i+1}-1}{2^{i-1}}\leq 0$, or equivalently, $\frac{e_{i+1}-1}{2^{i-1}} \leq R_i$. But $-\frac{1}{2^{i-1}} \leq \frac{e_{i+1}-1}{2^{i-1}} \leq 0$ since $e_{i+1} \in \{0,1 \}$ for $i \geq 2$, and $0 \leq R_i \leq 2^{1-i}$ by construction. Hence $$\frac{\alpha(r_i-1)}{n-r_i} \geq  \frac{1}{2+\theta}, $$ or $$\min\{\alpha(r_i-1), \alpha(n-r_i) \}=\alpha(r_i-1) \geq  \left(\frac{1}{2+\theta}\right) (n-r_i) $$ and (\ref{second_fraction}) holds.\\



Thus the result follows.\\

Finally, the result follows in the case $\theta = \infty$ since for all $n \geq 1$, we have $\alpha(n) \geq 1$, and so $\frac{\alpha(n)}{n} \geq \frac{1}{n}>0$.\\
 

 
\end{proof}

\begin{Lemma} \label{smaller}
 Let $\alpha(n)=\ell_R(I_n/I_{n+1})$ and let $\varepsilon > 0$ be given. 
\begin{enumerate}[1.]
 \item Suppose $\theta < \infty$. Then there exists $N(\varepsilon)$ such that for all $i \geq N(\varepsilon)$ and all $n \in [r_{i},r_{i+1})$, we have that $\alpha(n) \leq (\frac{1}{2+\theta} + \varepsilon)n + \alpha(r_{N(\varepsilon)}-1)$. 
 \item Suppose $\theta = \infty$. Then there exists $N(\varepsilon)$ such that for all $i \geq N(\varepsilon)$ and all $n \in [r_{i},r_{i+1})$, we have that $\alpha(n) \leq \varepsilon n + \alpha(r_{N(\varepsilon)}-1)$. 
\end{enumerate}

\end{Lemma}

\begin{proof}

We first verify the result when $\theta < \infty$. \\

We will use the formula

  $$\alpha(n) = \alpha(r_i-1) + \min\{ \alpha(r_i-1), \alpha(n-r_i) \},$$\\

 and the observation that if \\
 
\begin{equation} \label{first_fraction_bigger}
 \alpha(r_i-1) \leq \left(\frac{1}{2+\theta} + \varepsilon \right)r_i 
\end{equation}
 
 and 
 
 \begin{equation} \label{second_fraction_bigger}
\min\{\alpha(r_i-1), \alpha(n-r_i) \} \leq \left(\frac{1}{2+\theta} + \varepsilon \right)(n-r_i) + \alpha(r_{N(\varepsilon)}-1),  
 \end{equation}

 then upon adding these two inequalities, we obtain
 
 $$\alpha(n) \leq \left( \frac{1}{2+\theta} + \varepsilon \right)n + \alpha(r_{N(\varepsilon)}-1).$$ \\

 Observe that by Corollary \ref{limit}, for every $\varepsilon > 0$, there exists $N(\varepsilon)$ such that for all $i \geq N(\varepsilon)$, $\varepsilon_i < \varepsilon$, where $\varepsilon_i$ is as in Lemma \ref{epsilon}. Thus for $ i\geq N(\varepsilon)$, we have $\frac{\alpha(r_i-1)}{r_i}<\frac{1}{2+\theta} + \varepsilon$, or 
 
 \begin{equation} \label{first_smaller}
\alpha(r_i-1)<\left(\frac{1}{2+\theta} + \varepsilon\right)r_i.  
 \end{equation}

 Let $ i\geq N(\varepsilon)$ and suppose that $n-r_i \geq r_i$. Then 
 
 \begin{equation} \label{smaller_2}
 \min\{\alpha(r_i-1), \alpha(n-r_i) \}=\alpha(r_i-1) < \left(\frac{1}{2+\theta}+\varepsilon\right)r_i \leq \left(\frac{1}{2+\theta}+\varepsilon\right)(n-r_i).
 \end{equation}
 
 Thus adding (\ref{first_smaller}) and (\ref{smaller_2}) gives that $$\alpha(n) \leq \left(\frac{1}{2+\theta} + \varepsilon\right)n \leq  \left(\frac{1}{2+\theta} + \varepsilon\right)n + \alpha(r_{N(\varepsilon)}-1) $$ \\
 
 and hence the result for this case.\\
 
 Suppose next that $ i\geq N(\varepsilon)$ and $n-r_i < r_i$. We will prove this case by an inductive argument. \\
 
 First assume that $i=N(\varepsilon)$ for the base case. Then 
 \begin{equation} \label{smaller_3}
  \min\{\alpha(r_i-1), \alpha(n-r_i) \}=\alpha(n-r_i) \leq   
  \alpha(r_{N(\varepsilon)}-1) \leq \alpha(r_{N(\varepsilon)}-1) + \left(\frac{1}{2+\theta}+\varepsilon\right)(n-r_i),
 \end{equation}

 whence adding (\ref{first_smaller}) and (\ref{smaller_3}) gives that $$\alpha(n) \leq \left(\frac{1}{2+\theta} + \varepsilon\right)n + \alpha(r_{N(\varepsilon)}-1) $$ and the result for the base case. \\
 
 Next, assume the inductive statement: that $n-r_i < r_i$, that $i>N(\varepsilon)$ and that for all $N(\varepsilon) \leq j <i$, and all $m \in [r_j,r_{j+1})$, we have $\alpha(m) \leq \left(\frac{1}{2+\theta}+\varepsilon\right)m + \alpha(r_{N(\varepsilon)}-1)$. Since $n-r_i \in [r_j,r_{j+1})$ for some $j<i$, we have by the inductive statement if $j \geq N(\varepsilon)$, 
 
 \begin{equation} \label{smaller_4}
 \min\{\alpha(r_i-1), \alpha(n-r_i) \}=\alpha(n-r_i) \leq \left(\frac{1}{2+\theta}+\varepsilon\right)(n-r_i) + \alpha(r_{N(\varepsilon)}-1) ,
 \end{equation}\\

and if $j<N(\varepsilon)$, then $\alpha(n-r_i) \leq \alpha(r_{N(\varepsilon)}-1)$, and (\ref{smaller_4}) still holds.\\
 
 Thus adding (\ref{first_smaller}) and (\ref{smaller_4}) gives that $$\alpha(n) \leq \left(\frac{1}{2+\theta} + \varepsilon\right)n + \alpha(r_{N(\varepsilon)}-1) $$ and the result for the inductive step.\\
 
 Hence $$\alpha(n) = \alpha(r_i-1) + \min\{\alpha(r_i-1),\alpha(n-r_i) \} \leq \left( \frac{1}{2+\theta}+\varepsilon  \right)n + \alpha(r_{N(\varepsilon)}-1)$$ for $n \geq r_{N(\varepsilon)}$.\\

  Next, we verify the result when $\theta = \infty$. \\
  
We will use the formula

  $$\alpha(n) = \alpha(r_i-1) + \min\{ \alpha(r_i-1), \alpha(n-r_i) \},$$\\

 and the observation that if \\

 \begin{equation} \label{first_fraction_biggerer}
 \alpha(r_i-1) \leq \varepsilon r_i 
\end{equation}
 
 and 
 
 \begin{equation} \label{second_fraction_biggerer}
\min\{\alpha(r_i-1), \alpha(n-r_i) \} \leq \varepsilon (n-r_i) + \alpha(r_{N(\varepsilon)}-1),  
 \end{equation}

 then upon adding these two inequalities, we obtain
 
 $$\alpha(n) \leq \varepsilon n + \alpha(r_{N(\varepsilon)}-1).$$ \\
 
 Observe that by Corollary \ref{limit}, for every $\varepsilon > 0$, there exists $N(\varepsilon)$ such that for all $i \geq N(\varepsilon)$, $\varepsilon_i < \varepsilon$, where $\varepsilon_i$ is as in Lemma \ref{epsilon}. Thus for $ i\geq N(\varepsilon)$, we have $\frac{\alpha(r_i-1)}{r_i}<\varepsilon$, or \\
 
 \begin{equation} \label{first_smallerer}
\alpha(r_i-1)<\varepsilon r_i.  
 \end{equation}\\

 Let $ i\geq N(\varepsilon)$ and suppose that $n-r_i \geq r_i$. Then \\
 
 \begin{equation} \label{smallerer_2}
 \min\{\alpha(r_i-1), \alpha(n-r_i) \}=\alpha(r_i-1) < \varepsilon r_i \leq \varepsilon (n-r_i).\\
 \end{equation}\\
 
 Thus adding (\ref{first_smallerer}) and (\ref{smallerer_2}) gives that $$\alpha(n) \leq \varepsilon n \leq  \varepsilon n + \alpha(r_{N(\varepsilon)}-1) $$ \\
 
 and hence the result for this case.\\
 
 Suppose next that $ i\geq N(\varepsilon)$ and $n-r_i < r_i$. We will prove this case by an inductive argument. \\
 
 First assume that $i=N(\varepsilon)$ for the base case. Then 
 \begin{equation} \label{smallerer_3}
  \min\{\alpha(r_i-1), \alpha(n-r_i) \}=\alpha(n-r_i) \leq   
  \alpha(r_{N(\varepsilon)}-1) \leq \alpha(r_{N(\varepsilon)}-1) + \varepsilon (n-r_i),
 \end{equation}

 whence adding (\ref{first_smallerer}) and (\ref{smallerer_3}) gives that $$\alpha(n) \leq \varepsilon n + \alpha(r_{N(\varepsilon)}-1) $$ and the result for the base case. \\
 
 Next, assume the inductive statement: that $n-r_i < r_i$, that $i>N(\varepsilon)$ and that for all $N(\varepsilon) \leq j <i$, and all $m \in [r_j,r_{j+1})$, we have $\alpha(m) \leq \varepsilon m + \alpha(r_{N(\varepsilon)}-1)$. Since $n-r_i \in [r_j,r_{j+1})$ for some $j<i$, we have by the inductive statement if $j \geq N(\varepsilon)$, 
 
 \begin{equation} \label{smallerer_4}
 \min\{\alpha(r_i-1), \alpha(n-r_i) \}=\alpha(n-r_i) \leq \varepsilon (n-r_i) + \alpha(r_{N(\varepsilon)}-1) ,
 \end{equation}\\

and if $j<N(\varepsilon)$, then $\alpha(n-r_i) \leq \alpha(r_{N(\varepsilon)}-1)$, and (\ref{smallerer_4}) still holds.\\
 
 Thus adding (\ref{first_smallerer}) and (\ref{smallerer_4}) gives that $$\alpha(n) \leq \varepsilon n + \alpha(r_{N(\varepsilon)}-1) $$ and the result for the inductive step.\\
 
 Hence $$\alpha(n) = \alpha(r_i-1) + \min\{\alpha(r_i-1),\alpha(n-r_i) \} \leq \varepsilon n + \alpha(r_{N(\varepsilon)}-1)$$ for $n \geq r_{N(\varepsilon)}$.

 \end{proof}

\begin{Corollary} \label{limitdone}
Let $\alpha(n)=\ell_R(I_n/I_{n+1})$.
\begin{enumerate}[1.]
 \item Suppose $\theta < \infty$. Then $\lim_{n \rightarrow \infty} {\frac{\alpha(n)}{n}} = \frac{1}{2+\theta}$.
 \item Suppose $\theta = \infty$. Then $\lim_{n \rightarrow \infty} {\frac{\alpha(n)}{n}} = 0$.
\end{enumerate}

\end{Corollary}

\begin{proof}
First suppose $\theta < \infty$. Dividing both sides of the formula in Lemma \ref{smaller}(a) by $n$, we see that for all $\varepsilon >0$, we have that $\limsup_{n \rightarrow \infty} {\frac{\alpha(n)}{n}} \leq  \frac{1}{2+\theta}+\varepsilon$. Hence $$\limsup_{n \rightarrow \infty} {\frac{\alpha(n)}{n}} \leq  \frac{1}{2+\theta}.$$ Now by Lemma \ref{bigger}(a), we see that $$\liminf_{n \rightarrow \infty}{\frac{\alpha(n)}{n}} \geq \frac{1}{2+\theta}.$$ Hence the result follows.\\

Next suppose $\theta = \infty$. Dividing both sides of the formula in Lemma \ref{smaller}(b) by $n$, we see that for all $\varepsilon >0$, we have that $\limsup_{n \rightarrow \infty} {\frac{\alpha(n)}{n}} \leq  \varepsilon$. Hence $$\limsup_{n \rightarrow \infty} {\frac{\alpha(n)}{n}} \leq  0.$$ Now by Lemma \ref{bigger}(b), we see that $$\liminf_{n \rightarrow \infty}{\frac{\alpha(n)}{n}} \geq 0.$$ Hence the result follows.\\
\end{proof}

\begin{Theorem} \label{secondset_of_limits}
Let $C \in \mathbb{R} \cap [0,\frac{1}{2}]$ be given. There exists a regular local ring $(R,\mathfrak{m}_R)$ of dimension 2 and a discrete, rank 1 valuation $\nu$ of the quotient field of $R$ dominating $R$, such that each of the functions $\ell_R(R/I_n)$ and $\ell_R(I_n/I_{n+1})$ is not a quasi-polynomial plus a bounded function for large integers $n$ and such that $\lim_{n \rightarrow \infty} {\frac{\ell_R(I_n/I_{n+1})}{n}} = C$ and $\lim_{n \rightarrow \infty} {\frac{\ell_R(R/I_{n})}{n^2}} = \frac{C}{2}$. 
\end{Theorem}

\begin{proof}
If $0<C\leq \frac{1}{2}$, then let $\theta = \frac{1}{C}-2<\infty$, and if $C=0$, then let $\theta=\infty$. In either case, let $\{e_j\}$ be the associated sequence given by Definition \ref{sequence}, let the sequence $\{r_i\}_{i \in \mathbb{N}}$ be given by (\ref{inductive_defn}), and let $\nu$ be the valuation associated to $\{r_i\}$ by Lemma 1. We have proved the first claim in 
Corollary \ref{secondnon-poly}.\\

When $0<C\leq \frac{1}{2}$, we have $\lim_{n \rightarrow \infty} {\frac{\alpha(n)}{n}} = \frac{1}{2+\theta} = C$ by Corollary \ref{limitdone}(a). When $C=0$, we have $\lim_{n \rightarrow \infty} {\frac{\alpha(n)}{n}} = 0 = C$ by Corollary \ref{limitdone}(b). Thus the second claim holds.\\

Notice that $\ell_R(R/I_{n}) = \sum_{j=0}^{n-1}{\ell_R(I_j/I_{j+1})} = \ell_R(R/I_1) + \sum_{j=1}^{n-1}{\ell_R(I_j/I_{j+1})}$. Hence $$\frac{\ell_R(R/I_{n})}{n^2} = \frac{\ell_R(R/I_1) + \sum_{j=1}^{n-1}{\ell_R(I_j/I_{j+1})}}{n^2} = \frac{\ell_R(R/I_1)}{n^2} + \frac{\sum_{j=1}^{n-1}{\ell_R(I_j/I_{j+1})}}{n^2}.$$  By Lemma 5.1 of \cite{DaleOlgaKia}, $\lim_{n \rightarrow \infty} {\frac{\ell_R(R/I_{n})}{n^2}} = \frac{C}{2}$ and the third claim holds.\\

\end{proof}

\begin{Corollary} \label{secondirrational_limits}
There exists a regular local ring $(R,\mathfrak{m}_R)$ of dimension 2 and a discrete, rank 1 valuation $\nu$ of the quotient field of $R$ dominating $R$, such that each of the functions $\ell_R(R/I_n)$ and $\ell_R(I_n/I_{n+1})$ is not a quasi-polynomial plus a bounded function for large integers $n$ and such that each of $\lim_{n \rightarrow \infty} {\frac{\ell_R(I_n/I_{n+1})}{n}}$ and $\lim_{n \rightarrow \infty} {\frac{\ell_R(R/I_{n})}{n^2}}$ are irrational (even transcendental) positive numbers.
 
\end{Corollary}

\begin{proof}
Since the proof of Theorem \ref{secondset_of_limits} does not depend on whether $C>0$ is rational, we may choose $C \in (0,\frac{1}{2}) \setminus \mathbb{Q}$ in the first line of the proof of Theorem \ref{secondset_of_limits}, and the result follows.
 
\end{proof}


\begin{thebibliography}{1000000000}

\bibitem{Abhyankar}
{S. Abhyankar}, ''On the valuations centered in a local domain'',  Amer. J. Math. 78 (1956), 321-348.


   \bibitem{Dalemult}
  S. D. Cutkosky,
  ''Multiplicities associated to graded families of ideals'',
    Algebra and Number Theory 7 (2013), 2059-2083, 

\bibitem{Dale_final_mult}
{S. D. Cutkosky}, ''Multiplicities of graded families of linear series and ideals.'',  {\tt arXiv:1301.5613}  (2013).

\bibitem{Dale_asymptotic_mult}
{S. D.~Cutkosky}, ''Asymptotic Multiplicities'',  {\tt arXiv:1311.1432v1}  (2013).


\bibitem{DaleSrinivas}
{S. D.~Cutkosky, V.~Srinivas}, ''On a problem of Zariski on dimensions of linear systems'',  Annals of Math. 137 (1993), 531-559.

 
\bibitem{DaleOlgaKia}
  S. D. Cutkosky, K. Dalili, O. Kascheyeva,
  ''Growth of rank 1 valuation semigroups'',
    Communications in Algebra 38 (2010), 2768 - 2789.

    \bibitem{Dale}
  S. D. Cutkosky, P.A.Vinh,
  ''Valuation semigroups of two dimensional local rings'',
    Proceedings of the London Math. Soc. 108 (2014), 350-384. 
   
\bibitem{GHK}
L. Ghezzi, H. T. H\`{a}, O. Kashcheyeva, {``Toroidalization of generating sequences in dimension two function fields''}, Journal of Algebra 301(2) (2006), 838-866.

\bibitem{GK}
L. Ghezzi, O. Kashcheyeva, {``Toroidalization of generating sequences in dimension two function fields of positive characteristic''}, Journal of Pure and Applied Algebra 209(3) (2007), 631-649.



\bibitem{Mustata}
M. Mustata,{``On multiplicities of graded sequences of ideals''}, Journal of Algebra 256 (2002), 229-249.

\bibitem{Spivakovsky}
M. Spivakovsky, {``Valuations in function fields of surfaces,''}, Amer. J. Math. 112 (1990), 107-156.

 

\end{thebibliography}
\end{document}